\documentclass[12pt]{amsart}

\usepackage{amsmath}
\usepackage{amscd}
\usepackage[mathscr]{euscript}
\usepackage[margin=1in]{geometry}
\usepackage{amssymb}
\usepackage{amsthm}
\usepackage{tikz}
\usepackage{comment,faktor}

\theoremstyle{plain}
\numberwithin{equation}{section}
 \newtheorem{theorem}{Theorem}[section]
 \newtheorem{proposition}[theorem]{Proposition}
 \newtheorem{lemma}[theorem]{Lemma}
 \newtheorem{corollary}[theorem]{Corollary}
 
 \newtheorem{definition}[theorem]{Definition}

\newcommand{\C}{{\mathbb C}}
\newcommand{\Po}{{\mathbb P}}

\begin{document}

\title{Projective Joint Spectra and Characters of representations of $\tilde{A}_n$}

\author{T. Peebles}
\address{Department of Mathematics and Computer Science \\ Mesa Community College\\Mesa, AZ 85207}
\email{thomas.peebles@mesacc.edu}
%\email{mstessin@albany.edu}
\author{M. Stessin}
\address{Department of Mathematics and Statistics \\ University at Albany \\ Albany, NY 12222}
\email{mstessin@albany.edu}

\maketitle

\begin{abstract}
For a tuple of square complex-valued $N\times N$ matrices $A_1,\dots,A_n$ the determinant of their linear combination $x_1A_1+\cdots +x_nA_n$, which is called \textit{a pencil}, is a homogeneous polynomial of  degree  $N$ in $\C[x_1,...x_n]$. Zero-set of this polynomial is an algebraic set  in the projective space $\C\Po^{n-1}$. This set is called the determinantal hypersurface or determinantal manifold of the tuple $(A_1,...,A_n)$.  It was shown in \cite{CST} that determinantal hypersurfaces contain substantial information about representations of finite Coxeter groups. Namely, if $G$ is a non-special Coxeter group of type $A,B$, or $D$, $\rho_1$ and $\rho_2$ are two linear representations of $G$, and the determinantal hypersurfaces of images of the Coxeter generators of $G$ under $\rho_1$ and $\rho_2$ coincide as divisors in the projective space, the characters of $\rho_1$ and $\rho_2$ are equal, and, therefore, $\rho_1$ and $\rho_2$ are equivalent. 

In \cite{PST} this result was extended in the characters part to affine Coxeter groups of types $B,C$, and $D$. It was shown there that each such group contains a finite subset such that, if the determinantal hypersurfaces of the images of this set under two finite-dimensional representations coincide as divisors in the projective space, the characters of these representations are equal. Notably, the affine Coxeter groups of $A$ type are not covered by this result, as their combinatorics is quite different.

In this paper we explicitly construct a finite set in $\tilde{A}_n$ having the same property. We also show that every group which is a semidirect product of a fine group and a finitely generated abelian group contains a finite subset with the similar property:  for every finite-dimensonal representation of the group, the determinantal hypersurface of images of the set determines the representation character. 
\end{abstract}

\section{Introduction}

For a tuple of square complex-valued $N\times N$ matrices $A_1,\dots,A_n$ the determinant of their linear combination $x_1A_1+\cdots +x_nA_n$, which is called \textit{a pencil}, is a homogeneous polynomial of  degree  $N$ in $\C[x_1,...x_n]$. Zero-set of this polynomial is an algebraic set  in the projective space $\C\Po^{n-1}$. This set is called the \textbf{determinantal hypersurface} or \textbf{determinantal manifold} of the tuple $(A_1,...,A_n)$ and is denoted by $\sigma(A_1,...,A_n)$:
$$\sigma(A_1,...,A_n)=\Big\{ [x_1:...:x_n]\in \C\Po^{n-1}:  \ det(x_1A_1+\cdots + x_nA_n)=0 \Big\}.$$
An infinite-dimesional analog of the determinantal hypersurfaces was introduced in \cite{Y} and called \textbf{the projective joint spectra}. Given a tuple $(A_1,...,A_n)$ of linear operators acting on a Hilbert space $H$, the projective joint spectrum of the tuple is 
$$\sigma(A_1,...,A_n)=\Big\{[x_1:...:x_n]\in \C\Po^{n-1}: \ x_1A_1+\cdots +x_nA_n \ \mbox{is not invertible} \Big\} . $$
When $H$ is finite-dimensional, the projective joint spectrum and determinantal manifold are the same, so in this case we use these names interchangably. 

In general, it is possible that $\sigma(A_1,...,A_n)=\C\Po^{n-1}$. To avoid such redundancy  it is frequently assumed that at least one of the elements in the tuple is invertible and usually taken to be the identity matrix (operator). We will assume that the last operator in the tuple is the identity. In this case, the determinantal manifold has codimension 1, which justifies the name hypersurface. It was shown in \cite{ST} that in infinite-dimensional case if $[x_1:...:x_n:x_{n+1}]\in \sigma(A_1,...,A_n,I),$ and $(-x_{n+1})$ is an isolated spectral point of $x_1A_1+\cdots +x_nA_n$ of finite multiplicity, then the projective joint spectrum $\sigma(A_1,...,A_n,I)$ is an analytic set of codimension 1 in a neighborhood  of $ [x_1:...:x_n:x_{n+1}]$.

In what follows we need the following 2 specifications of the projective joint spectrum.

For a tuple $(A_1,...,A_n)$ we call \textbf{proper projective joint spectrum} of $(A_1,...A_n)$ the part of the projective joint spectrum $\sigma(A_1,...,A_n,I)$ that lies in the chart $\Big\{ x_{n+1}\neq 0\Big\}$ and we set $x_{n+1}=-1$. We denote the proper projective joint spectrum by $\sigma_p(A_1,...,A_n)$ and consider it as a set in $\C^n$, so that in the finite-dimensional case
\begin{equation}\label{proper spectrum}
\sigma_p(A_1,...,A_n)=\Big\{ (x_1,...,x_n): \ det(x_1A_1+\cdots +x_nA_n-I)=0\Big\}. \end{equation} 
It follows from \eqref{proper spectrum} that $1 \in \sigma(x_1A_1+\cdots+x_nA_n)$ for every point $(x_1,x_2,\ldots,x_n) \in \sigma_p(A_1,\ldots, A_n)$. Since we consider the finite-dimensional case, 1 is an eigenvalue of $x_1A_1+\dots+x_nA_n$,  and we can choose a small contour $\gamma$ around 1 with no other eigenvalues inside.  If $(x_1,...,x_n)$ is a regular point of $\sigma_p(A_1,...,A_n)$, the rank of the projection 
$$ P(x_1,\dots,x_n)=\frac{1}{2\pi i} \int_{\gamma} \bigg( w-x_1A_1-\dots -x_nA_n\bigg)^{-1}dw$$
is the multiplicity  of this eigenvalue (that is the sum of dimensions of all Jordan cells corresponding to eigenvalue 1).  It is well-known and is easy to prove that this multiplicity is the same at every pair of regular points belonging to the same component of the algebraic set $\sigma_p(A_1,...,A_n)$. We prescribe this multiplicity to the component and define the joint spectrum in the divisor form, $\sigma_p^d(A_1,...,A_n)$  as the divisor in ${\mathbb C}^n$ which is the sum of spectral components with these prescribed multiplicities. Of course, it coincides with the zero-divisor of the polynomial $det(x_1A_1+\dots+x_nA_n-I)\in {\mathbb C}[x_1,\dots,x_n]$.

The study of determinants of matrix pencils started in late 1800-s. In the case when $G$ is a finite group and the tuple is the set of transformations of the group-algebra of $G$ given by all the  elements of $G$, the corresponding determinant is called \textbf{the group determinant}. It's study ultimately led Frobenius \cite{Fr1} - \cite{Fr3} to create the foundation of representation theory. A systematic study of the question when a hypersurface in $\C\Po^n$ has a determinantal representation was originated by Dickson \cite{D1} - \cite{D5} and a lot of work has been produced in this area since. Not trying to give an exhaustive account of the results here we just mention \cite{C}, \cite{D1}-\cite{D5}, \cite{Dol}, \cite{HKV}, \cite{Helt}-\cite{Helt2}, \cite{KV}, \cite{V}  and references there.

Another natural question is: given that a hypersurface in $\C\Po^n$ admits a determinantal representation, what does the geometry of the surface tell us about relations between operators in a corresponding tuple? This question was not paid much attention for a long time. The only result in this direction published before 2000-s, which we know of,  is the 1952 paper by Motzkin and Taussky \cite{MT}. The situation changed in the last 10-15 years, and a lot of research investigating determinantal manifolds and projective joint spectra from this angle has been produced during this period. Here we mention papers \cite{Am}, \cite{BCY}, \cite{CY}, \cite{CSZ}-\cite{CST}, \cite{DY}, \cite{GY1}, \cite{GY2}, \cite{KV}, \cite{PS}, \cite{PST}, \cite{R}, \cite{St} \cite{ST}, \cite{SYZ}, \cite{Y} and references there.

In particular, it was shown in \cite{CST} that determinantal hypersurfaces contain substantial information about representations of finite Coxeter groups.  

Given a finite group $G$, a subset $L=\{g_1,...,g_k\}\subset G$, and a linear representation $\rho:G\to GL(V)$ we write
$$\sigma_p^d(L,\rho)=\sigma_p^d\Big(\rho(g_1),...,\rho(g_n)\Big). $$ 
 It follows  from the proof of Theorem 1.1 in \cite{CST} ( see Lemma \ref{Powers lemma} below) that, if a subset $L$ represents all conjugacy classes of $G$ and two representations of $G$ satisfy
 $$  \sigma_p^d(L,\rho_1)=\sigma_p^d(L,\rho_2), $$
 then the characters of $\rho_1$ and $\rho_2$ are the same and, therefore, these representations are equivalent. It was also shown in \cite{CST} that for finite non-special Coxeter groups of types A,B, and D the set $L$ can be substantially smaller: just consisting of the set of Coxeter generators of the group. These results show that for finite groups proper projective joint spectra in the divisor form carry at least as much information about representations as characters.
 
 Thus, it is natural to ask: \\
 a) whether the projective joint spectrum in the divisor form of images under a representation of every generating set of a finite group $G$  determines the representation; \\
 b) for which infinite groups can we find a finite subset $L$ such that the $\sigma_p^d(L,\rho)$ determines the character of $\rho$ for every finite-dimensional representation? For such groups the projective joint spectra carry all the information about representations provided by characters.
 
 The answer to the first question is negative. The first counterexample was given in \cite{KV} and additional ones related to Hadamard matrices and representations of certain subgroups of permutation groups can be found in \cite{PS}.
 
 As for the second question,  there was an explicit construction in \cite{PST} showing that for affine Coxeter groups of types $\tilde{B},\tilde{C}$, and $\tilde{D}$ there exist finite subsets whose proper projective joint spectra in the divisor form determine characters of all finite-dimensional representations. Notably, $\tilde{A}$ type groups are not covered by this result as their combinatorics is different.
 
In this paper we fill this gap and explicitly construct a finite set in $\tilde{A}_n$ for which projective joint spectra in the divisor form of it's image under linear representations determine representations' characters (Theorem \ref{main new}). We also show that the answer to question b) above is positive for groups, which we call \textbf{affine type groups} (Theorem \ref{A-groups}). These are semidirect products of finite groups and finitely generated abelian groups.

The structure of this paper is as follows.  In Section \ref{section 2} we collect necessary background results. Section \ref{section 3} is devoted to a construction of finite sets in affine type groups that positively answers question b) for such groups. Section \ref{subgroupconstruction} is devoted to combinatorics of $\tilde{A}_n$ and deriving $\tilde{A}$-echelon forms. Finally, in  Section \ref{main} we give an explicit construction of the finite set in $\tilde{A}_n$ whose spectra determine characters of finite-dimensional representations.

\section{Background and combinatorics of $A_n$}\label{section 2}
\subsection{Projective joint spectra} \label{PJSbackground}

In what follows we will need the following Lemma proved in \cite{CST}.

\begin{lemma}\label{Powers lemma} 

Let  $A_1, \cdots, A_n$, and $B_1, \cdots, B_n$  be linear operators acting on a finite-dimensional Hilbert space $H$.  If 
$$\sigma_p^d(A_1, \cdots,A_n) = \sigma_p^d(B_1,\cdots, B_n)$$
then for every $m = (m_1,m_2,\cdots,m_n)\in {\mathbb N}^n$ we have 

$$\displaystyle \sum_{(j_1,\ldots,j_k), sign(A_{j_1}A_{j_2}\cdots A_{j_k}) = m} Tr(A_{j_1}A_{j_2}\cdots A_{j_k})$$ 

$$ =\sum_{(j_1,\ldots,j_k), sign(B_{j_1}B_{j_2}\cdots B_{j_k}) = m} Tr(B_{j_1}B_{j_2}\cdots B_{j_k})$$
where the sum is taken over all products with the specified signature. 

\end{lemma}

\subsection{Group representations}
Recall that a representation of a group $G$ is a homomorphism $\rho: G \rightarrow GL(V)$, from $G$ to the group of bounded invertible linear operators acting on a linear space $V$.  For finite-dimensional representations an important functional used in representation theory is the character of a representation, that is defined as follows.
\begin{definition}
Given a representation, $(\rho, V)$ of a group $G$, its \textbf{character}, denoted by $\chi_{\rho}$, is defined as
$$\chi_\rho(g) = Tr(\rho(g))$$
\end{definition}
 
Two representations $(\rho_1,V_1)$ and $(\rho_2,V_2)$ of a group $G$ are called \textbf{equivalent} (the notation is $\rho_1 \sim \rho_2$), if there is an isomorphism $C : V_1\to V_2$ such that for all $g \in G$ we have
$$\rho_1(g) = C^{-1}\rho_2(g)C$$
The following result is well known: 
\begin{theorem}\label{character theorem}
Let $G$ be a finite group. If two representations $\rho_1, \rho_2$ of $G$ satisfy 
$$\chi_{\rho_1} = \chi_{\rho_2}$$
then $\rho_1 \sim \rho_2$.
\end{theorem}
This and other facts about linear representations of groups could be found in \cite{S}.

Given a set of generators of a group $G$, this set is considered as an alphabet, and every element $w$ in $G$ is represented by a word $\tau$ in this alphabet (of course, this representation may not be unique). The length of a word $\tau$, \  $|\tau|$, is the number of letters it has.  The length of an element $w\in G$ is the smallest length of a word representing this element.  See \cite{GP} for a more detailed account of this.  

\subsection{Coxeter Groups}\label{CGbackground}
A Coxeter group is a finitely generated group that has the presentation 
$$
\langle
s_i\in S \mid (s_is_j)^{m_{ij}} 
= 1 \ \text{and} \ m_{ii} = 1
\rangle
$$ 
where $S=\{s_1,\dots, s_n\}$ is a finite set of generators and
$m_{ij}\in \{1, 2, \ldots, \infty\}$. 

  The matrix $M_{i,j} = [m_{ij}]$ is called the Coxeter matrix, 
   and, in order to avoid redundancies, it is always assumed  that this matrix is symmetric. It is easy to see that if $m_{i,j}=2$, then $s_i$ and $s_j$ commute.
 
Another way to express a Coxeter group is through its Coxeter diagram.  This is  a graph where each vertex represents a generating element and two vertices, say $i$, $j$, are connected by an edge if $m_{i,j} \geq 3$, and, if $m_{i,j} > 3$, then this number is placed above the connecting edge.  In this paper we are interested in the affine Coxeter groups of type $\tilde{A}_n$ and the finite Coxeter groups  $A_n$ which are present by the following Coxeter diagrams:

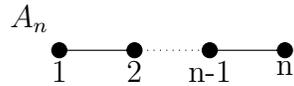
\begin{center}
\begin{tikzpicture}

\pgfmatrix{rectangle}{center}{Coxeter Group $A_n$}{}{}{}
{
\node[left]{$A_n$}; \\ 
\draw[fill=black](0,0) circle[radius = .1] node[below]{1} -- (1,0)circle[radius = .1]node[below]{2};
\draw[dotted](1,0)--(2,0); 
\draw[fill = black](2, 0) circle[radius = .1] node[below]{n-1}--(3,0)circle[radius = .1] node[below]{n};\\
}
\end{tikzpicture}

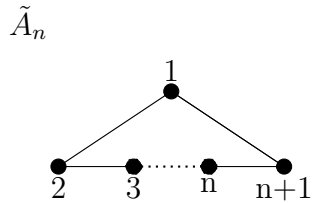
\begin{tikzpicture}
\pgfmatrix{rectangle}{center}{Affine Coxeter Group}{}{}{}
{
\node[left]{$\tilde{A}_n$} ; \\
\draw[fill=black](0,0) circle[radius = .1] node[below]{2} -- (1,0) circle[radius = .1]node[below]{3};
\draw[dotted, thick] (1,0) circle[radius = .1] -- (2,0)circle[radius = .1]node[below]{n};
\draw[fill=black] (2,0) circle[radius = .1] -- (3,0) circle[radius = .1] node[below]{n+1};
\draw[fill = black] (0,0) -- (1.5,1) circle[radius = .1]node[above]{1};
\draw[fill = black](3,0) -- (1.5,1);\\
}
\end{tikzpicture}
\end{center}

Notice $\tilde{A}_n$ has $n+1$ vertices and we enumerate the top element as 1 and count counterclockwise.  Removing any vertex, we get a group that is isomorphic to $A_n$. 

More information on Coxeter groups, their properties, and combinatorics may be found in \cite{BB,HH,JH}.

\subsection{Admissible transformations and the combinatorics of $A_n$}

In order to explicitly represent a decomposition of $\tilde{A}_n$  into a semidirect product we will need to have a better understanding of the combinatorics of $A_n$.  We use letter ``a" for Coxeter generators of $A_n$ and $\tilde{A}_n$. It was shown in \cite{GP} that using \textbf{admissible transformations}, every element of $A_n$ can be reduced to a canonical form. In \cite{CST} this canonical form for groups $A_n,B_n$ and $D_n$, that was called in \cite{CST}  \textbf{echelon form}, was obtained via an algorithm that used a more restricted set of admissible transformations.  The admissible transformations of words are: 

\begin{enumerate}
\item Cancelling transformations -  
%for each $i$
$$w'a_ia_iw'' \rightarrow w'w''$$
\item Commuting transformations - if $a_i$ and $a_j$ commute then the transformation is
$$w'a_ia_jw'' \rightarrow w'a_ja_iw''$$
\item Circular transformation - given a word $a_{i_1}\cdots a_{i_k}a_{i_{k+1}}\cdots a_{i_N}$ we have
$$a_{i_1}\cdots a_{i_k}a_{i_{k+1}}\cdots a_{i_N} \rightarrow a_{i_{k+1}}\cdots a_{i_N}a_{i_1}\cdots a_{i_k}$$
\item Replacement transformations - This replaces a certain subword consisting of a string of letters by another representation of this subword.  
$$w'a_ia_{i+1}a_iw'' \rightarrow w'a_{i+1}a_ia_{i+1}w''$$
\end{enumerate}

The following result was proved in \cite{GP} (see also \cite{CST}):
\begin{theorem}
Every word, $w \in A_n$, can be transformed using admissible transformations into a word 
\begin{equation}\label{echelon}
\tilde{w}=\delta_1\delta_2\cdots \delta_n
\end{equation}
where $\delta_i \in \{1,a_i\}$. This algorithm does not increase the length of the word.
\end{theorem}

The word \eqref{echelon} is called \textbf{an echelon form} of $w$.

 It is important to notice that admissible transformations  preserve conjugacy classes, and, therefore, characters of representations.  
 
 Every non-trivial echelon form \eqref{echelon}  can be partitioned into commuting blocks consisting of consecutive letters $a_j$. They are separated by
 elements  $\delta_k = 1$ appearing after the first non-trivial letter. We denote these blocks by $\Delta_j$ and obtain the following form of the word
 % what we call  \textit{block echelon form } that looks like 
 \begin{equation} \label{A-echelon}
 \Delta_1, \Delta_2,...,\Delta_k ,
 \end{equation}
 
   If there are two consecutive blocks, say 
   $$\Delta_1=a_pa_{p+1}\dots a_j, \ \mbox{and} \ \Delta_2=a_ra_{r+1}\dots a_s,$$
 with $r-j\geq 3$, then $a_{j+1}$ commutes with all generators in blocks $\Delta_2,\dots, \Delta_k$, as well as with $a_p,\dots, a_{j-1}$ so we can use the fact that $a_k^2=1$ and perform commuting and circular transformations to obtain 
 \begin{eqnarray*}
& \Delta_1 \Delta_2 \dots=a_p\dots a_j \boxed{a_{j+1}a_{j+1}} a_r	\dots a_s \dots \sim a_p\cdots a_{j-1} a_{j+1}a_ja_{j+1}a_r\dots a_s\dots \\
& = a_p\cdots a_{j-1} a_ja_{j+1}a_ja_r\dots a_s\dots \sim a_p\dots a_ja_{j-1}a_ja_{j+1}a_r\dots a_s\dots \\
& \cdots \cdots \cdots \cdots \cdots \cdots \cdots \cdots \cdots \cdots \cdots \cdots \cdots \cdots \cdots \cdots \cdots \cdots \cdots \cdots \\
&\sim a_{p+1}a_pa_{p+1}a_{p+2}a_{p+3}\cdots a_{j+1} a_r\cdots a_s \cdots= a_pa_{p+1}a_pa_{p+2}\cdots a_{j+1}a_r\cdots a_s\cdots \\
&\sim \boxed{a_{p+1}\cdots a_{j+1}}\boxed{a_r \cdots a_s}\cdots=\tilde{\Delta}_1\Delta_2 \cdots .
 \end{eqnarray*}
We resulted in a new block form with the distance between the first two blocks, being $r-j-1$, that is  1 less than in the original representation. Continuing this process we will eventually come to the block form with the distance between the end of the first block and the beginning of the second one being 2, meaning that there is only one letter gap between them.

Applying this process to all adjacent blocks, we will come to a block form with 1 space between every two consecutive non-trivial blocks. We call such block form \textbf{block echelon form}. It will be used in the sections \ref{subgroupconstruction} and \ref{main}.

%%%%%%%%%%%%%%%%%%Characters of Affine type groups%%%%%%%%%%%%%%

\section{Characters of representations of affine type groups}\label{section 3}

In this section we answer in affirmative to the question b) in introduction for a certain type of groups. We call these groups \textbf{affine type groups}. They resemble affine Coxeter groups, and that is why the name.

\begin{definition}
We call a group $G$ an \textbf{affine type group}, or $A$-type group, if it has a finitely generated normal abelian subgroup $G_1\subset G$ such that the quotient $G/G_1$ is a finite group.	
\end{definition}

Of course, every finite group is of $A$-type ( in this case $G_1=\{1\}$). Also, non special affine Coxeter groups are of this type. A specific construction of an abelian normal subgroup for $\tilde{A}_n$ is given in section \ref{subgroupconstruction}. For groups $\tilde{B}_n, \ \tilde{C}_n$, and $\tilde{D}_n$ a descriptiopn of corresponding normal abelian subgroups could be found in \cite{PST}.

The following result is a direct corollary to Lemma \ref{Powers lemma} and Theorem \ref{character theorem}.

Let $G$ be an $A$-type group, \ $G_1$ be a normal abelian subgroup of $G$, \ $h_1,...,h_m\in G$ be a set of generators of $G_1$, and let $g_1,...,g_k\in G$ be a set which represents every conjugacy class of $G/G_1$, that is for every conjugacy class of $G/G_1$ there is $g_j$ such that $g_jG_1$ belongs to this class. Write 
\begin{equation}\label{K} 
\mathscr{K}=\{g_1,...,g_k,h_1,...,h_m,h_1^{-1},...,h_m^{-1}\} 
\end{equation}

\begin{theorem}\label{A-groups}
Let $G$ be an $A$-type group	,  $G_1\subset G$ be a normal abelian subgroup of $G$, and ${\mathscr K}$ be given by \eqref{K}. If two finite dimensional linear representations of $G$, \ $\rho_1$ and $\rho_2$, satify,
\begin{equation}\label{A type}
	\sigma_p^d({\mathscr K},\rho_1)=\sigma_p^d({\mathscr K},\rho_2), 
	\end{equation}
then
$$\chi_{\rho_1}=\chi_{\rho_2}. $$  
\end{theorem}

\begin{proof}
Let $g\in G$. There are $j\in \{1,...,k\}$, $w\in G$ and $u\in G_1$ such that $g=wg_jw^{-1}	u$.
Since circular transformations preserve conjugacy classes, we have
$$g \sim g_j\Big(w^{-1}uw\Big) .$$
Since $G_1$ is normal $w^{-1}uw\in G_1$, and, therefore, there are $\ell_1,...,\ell_m \in {\mathbb Z}$ such that 
$$w^{-1}uw=h_1^{\ell_1}...h_m^{\ell_m}, \ \chi_{\rho_i}(g)=\chi_{\rho_i} (g_jh_1^{\ell_1}...h_k^{\ell_k}), \ i=1,2.$$
Write ${\mathcal K}(j,\ell_1,...,\ell_m)=\{g_j,h_1^{sign(\ell_1)},...,h_k^{sign(\ell_m)}\} $. Equation \eqref{A type} implies that
$$\sigma_p^d\Big( {\mathcal K}\big(j,\ell_1,...,\ell_m\big),\rho_1\Big)=\sigma_p^d\Big( {\mathcal K}\big(j,\ell_1,...,\ell_m\big),\rho_2\Big). $$
Consider a word consisting of 1 letter $g_j$ and $|\ell_r|$ letters  $h_r^{sign(\ell_r)}, \ r=1,...,m$. This word has  form $w_1g_jw_2$, where $w_1$ and $w_2$ are words comprised of letters $h_r^{sign(\ell_r)}, \ r=1,...,m$ with the total number of $h_r^{sign(\ell_r)} $ in both words being equal to $|\ell_r|$. Using a circular transformation we see that
 $$w_1g_jw_2 \sim g_jw_2w_1=g_jh_1^{\ell_1}...h_k^{\ell_m},$$
where the last equality follows from the fact that $G_1$ is abelian. It now follows from Lemma \ref{Powers lemma} that $\chi_{\rho_1}(g)=\chi_{\rho_2}(g).$
\end{proof}

\vspace{.2cm}

Of course, we can chose $g_1,...,g_k$ by picking up one element from each coset of $G/G_1$ and obtain the following    
\begin{corollary}
Let $G$ be an $A$-type group, $G_1$ be a finitely generated abelian normal subgroup such that $G/G_1$ is finite, and $h_1,...,h_m$ be generators of $G_1.$ If
$$\sigma_p^d\bigg( \Big(\faktor{G}{G_1},h_1,...,h_m,h_1^{-1},...,h_m^{-1}\Big),\rho_1\bigg)=\sigma_p^d\bigg( \Big(\faktor{G}{G_1},h_1,...,h_m,h_1^{-1},...,h_m^{-1}\Big),\rho_2\bigg),$$	
then
$$\chi_{\rho_1}=\chi_{\rho_2}. $$
\end{corollary}

For some groups the set $g_1,...,g_k$ in Theorem \ref{A-groups} can be reduced to a smaller set which does not represent all conjugacy classes of $G/G_1$, but the same conclusion  still holds. This is true for non-special finite Coxeter groups (see \cite{CST}) and for affine Coxeter groups of types $\tilde{B}, \ \tilde{C}$, and $\tilde{D}$ (see \cite{PST}). In the rest of this paper we will show that a similar reduction is possible for the affine Coxeter groups of type $\tilde{A}$.

%%%%%%%%%%%%%%%%%%% Combinatorics of tildeA_n%%%%%%%%%%%%%%%%%%%
\section{Special elements and combinatorics of $\tilde{A}_n$}\label{subgroupconstruction}

\subsection{Combinatorics of $\tilde{A}_n$}
We build up the elements in $\tilde{A}_n$ to form generators of a normal abelian subgroup.  Consider the following elements: 
\begin{align*}
p_1 &= a_2a_3\cdots a_na_{n+1}\\
p_2 &= a_3a_4\cdots a_na_{n+1}a_1\\
& \ \, \vdots \\
p_i &= a_{i+1}\cdots a_{n}a_{n+1}a_1\cdots a_{i-1}\\
& \ \, \vdots &\\ 
p_n &= a_{n+1}a_1\cdots a_{n-2}a_{n-1}\\
p_{n+1} &= a_1\cdots a_n
\end{align*}

Because of the circular structure of the Coxeter diagram, we will use  $mod(n+1)$ arithmetic whenever we are dealing with indices.

First, we establish basic commuting relations between the generating elements $a_i$ and the elements $p_j$. 

\begin{proposition} \label{ai pi relations}
Let $1 \leq i,j \leq n+1$, then 
$$a_ip_j = \left\{ \begin{matrix} p_{j-1}a_{i-1} & \text{if} \ i = j \\
p_{j+1}a_{i-1} & \text{if} \  j+1 = i \\ 
p_ja_{i-1} & \text{otherwise}
\end{matrix} \right.$$
where, as mentioned above, all indexes are considered $mod(n+1)$.
\end{proposition}

\begin{proof}
(Case 1) $i=j$
We have 
\begin{align*}
a_ip_i & = a_i (a_{i+1}\cdots a_{n}a_{n+1}a_1\cdots a_{i-1})\\
& = (a_ia_{i+1}\cdots a_{n}a_{n+1}a_1\cdots a_{i-2}) a_{i-1}\\
& = p_{i-1}a_{i-1}
\end{align*}

(Case 2) $ j+1=i \mod (n+1) $
Then again by definition:
\begin{align*}
a_ip_j & = a_i(a_ia_{i+1}\cdots a_{i-2}) \\
& = a_{i+1} \cdots a_{i-2} \\
& = a_{i+1} \cdots a_{i-2}\cdot a_{i-1} a_{i-1}\\
&= p_{j+1}a_{i-1}
\end{align*}

(Case 3) $i < j$.
In this case, we use the relation: 
$$a_ia_{i+1}a_i = a_{i+1}a_ia_{i+1}$$
which follows from $(a_ia_{i+1})^3 = 1$.  
Now we have: 
\begin{align*}
a_ip_j &= a_i(a_{j+1}\cdots a_{n+1}a_1a_2\cdots a_{i-1}a_ia_{i+1}\cdots a_{j-1})\\
&= a_{j+1}\cdots a_{n+1}a_1a_2\cdots{\bf a_i} a_{i-1}a_ia_{i+1}\cdots a_{j-1}\\
&= a_{j+1}\cdots a_{n+1}a_1a_2\cdots{\bf a_{i-1}a_{i}a_{i-1}}a_{i+1}\cdots a_{j-1}\\
&= a_{j+1}\cdots a_{n+1}a_1a_2\cdots a_{i-1}a_ia_{i+1}\cdots a_{j-1} {\bf a_{i-1}}.
\end{align*}

The case $i > j+1$ is handled in a similar way and is left out. 
\end{proof}

\vspace{.3cm}

Passing to the inverses  in relations of Proposition \ref{ai pi relations} and recalling $a_i^{-1} = a_i$ we get
\begin{corollary} \label{ai pi inverse relations}
Let $1 \leq i,j \leq n+1$, then 
$$a_ip_j^{-1} = \left\{ \begin{matrix} p_{j+1}^{-1}a_{i+1} & \text{if} \ i = j \\
p_{j-1}^{-1}a_{i+1} & \text{if} \  j+1 = i \\ 
p_j^{-1}a_{i+1} & \text{otherwise}
\end{matrix} \right.$$

\end{corollary}

\vspace{.3cm}

\begin{proposition} \label{pi relations}
Let $1 \leq i,j \leq n+1$, then  
$$p_ip_j = p_{j+1}p_{i-1}$$
\end{proposition}

\begin{proof}
We split the proof into two cases.

(Case 1) $j = i-1 \mod (n+1)$  The equality holds trivially.

(Case 2) $j \neq i-1 \ mod(n+1).$
Then either $1 \leq j < i-1$, or $i+1 \leq j \leq n+1$.  Both cases are treated similarly so we give a proof for  the first one.
\begin{align*}
p_ip_j &= a_{i+1}\cdots a_na_{n+1}a_1\cdots a_ja_{j+1}\cdots a_{i-2}a_{i-1} p_j\\
& = a_{i+1}\cdots a_na_{n+1}a_1\cdots a_ja_{j+1}\cdots a_{i-2}p_ja_{i-2}
\end{align*} 
Using Proposition \ref{ai pi relations} we continue to move $p_j$ to the left and get

\begin{align*}
& = a_{i+1}\cdots a_na_{n+1}a_1\cdots a_ja_{j+1}\cdots a_{i-2}{\bf p_j}a_{i-2}\\
&= a_{i+1}\cdots a_na_{n+1}a_1\cdots a_ja_{j+1}\cdots {\bf p_j}a_{i-3}a_{i-2}\\
& \ \, \vdots \\
&= a_{i+1}\cdots a_na_{n+1}a_1\cdots a_ja_{j+1}{\bf p_j}\cdots a_{i-3}a_{i-2}\\
&=a_{i+1}\cdots a_na_{n+1}a_1\cdots a_j{\bf p_{j+1}}a_{j}\cdots a_{i-3}a_{i-2},\\
& = p_{j+1} a_i a_{i+1}\cdots a_{n+1}\cdots a_{i-1}\\
& = p_{j+1}p_{i-1}
\end{align*}
 
(remind that all indexes are $mod(n+1)$ numbers).  This finishes the proof. 

\end{proof}
Again, passing to inverses we obtain
\begin{corollary} \label{inverse pi relations}
The following relations hold:
$$p_j^{-1}p_i^{-1} = p_{i-1}^{-1}p_{j+1}^{-1} \quad  p_j^{-1}p_i = p_{i-1}p_{j-1}^{-1}$$
\end{corollary}

%%%%%%%%%%%%%%%%%% Construction of NASG for tilde An %%%%%%%%%%%%%%%%%%%%

\subsection{Construction of a normal abelian subgroup}

 Consider the set
 $$P = \{p_j^{-1}p_i \mid 1 \leq i,j \leq n+1\}.$$
 Proposition \ref{ai pi relations} and Corollary \ref{ai pi inverse relations} imply that for every $k,i,j$ there exist $m$ and $l$ such that
$$ a_k(p_j^{-1}p_i) = p_{m}^{-1}a_{k+1}p_i =  p_{m}^{-1}p_{l}a_{k},$$
 so the subgroup generated by the set $P$ is normal.  Let $N = <P>$. It is easily seen that $N$ is generated by $(n+1)$ elements $p_j^{-1}p_{j+1}, \ 1\leq j\leq n+1$, as, if $i<j$,
 then
 $$p_i^{-1}p_j=(p_i^{-1}p_{i+1})(p_{i+1}^{-1}p_{i+2})\cdots  (p_{j-1}^{-1}p_j),$$
 and if $i>j$, then
 $$p_i^{-1}p_j=\big( p_{i-1}^{-1}p_i\big)^{-1}\big(p_{i-2}^{-1}p_{i-1}\big)^{-1}\cdots \big( p_j^{-1}p_{j+1}\big)^{-1}. $$
 
  Write 
  \begin{equation}\label{g}
  	g_j = p_j^{-1}p_{j+1}, \ 1 \leq j \leq n.
  	\end{equation} 
  The following Proposition follows directly from Proposition \ref{ai pi relations} and Corollary \ref{ai pi inverse relations}.  
 \begin{proposition} \label{ai gj relations}
 For each $1 \leq j \leq n$, and $k \neq j-1,j,j+1$ 
 $$a_kg_j = g_ja_k.$$
 Also, the following relations hold: 
 $$\begin{matrix} a_{i-1}g_i = g_{i-1}g_ia_{i-1} & a_ig_i = g_i^{-1}a_i & a_{i+1}g_i = g_ig_{i+1}a_{i+1} \end{matrix},$$
 
 where all indexes are considered $mod(n+1)$.
 \end{proposition}
 
  As a direct corollary to Proposition \ref{ai gj relations} we obtain the following result.
  
 \begin{corollary} \label{commutations delta and g}
 Let $\Delta$ be an $A_n$ echelon block with $|\Delta|>1$.  If $\Delta = a_pa_{p+1}\cdots a_{k-1}a_k$, then the following commutation relations hold
 \begin{center}
 \begin{tabular}{c c}
 $\Delta g_j = g_{j+1}\Delta$ & if $p \leq j < k$ \\
$\Delta g_k = g_p^{-1}g_{p+1}^{-1}\cdots g_k^{-1}  \Delta$ & if $j=k$ \\
$g_j\Delta = \Delta g_{j-1}$ & if $p<j \leq k$  \\ 
$g_p\Delta = \Delta g_p^{-1}\cdots g_k^{-1}$ & if $j=p$ \\
$g_j\Delta = \Delta g_j$ & if $j < p-1$ or $j>k+1$
 \end{tabular}
 \end{center}
 \end{corollary}

\begin{proof}
  For $j<p-1$ or $j>k+1$ then by Proposition \ref{ai gj relations}, $g_j$ commutes with every element in $\Delta$ so we get
$$g_j\Delta = \Delta g_j.$$

If $p\leq j < k$, then
\begin{align*}
\Delta g_j &= a_p \cdot a_ja_{j+1}\cdots a_kg_j \\
&= a_p\cdots a_ja_{j+1}g_j\cdots a_k \\
&= a_p\cdots a_jg_jg_{j+1}a_{j+1}\cdots a_k \\ 
&= a_p\cdots g_j^{-1}a_jg_{j+1}a_{j+1}\cdots a_k \\
&= a_p\cdots g_j^{-1}g_jg_{j+1}a_ja_{j+1}\cdots a_k\\
&= a_p\cdots g_{j+1}a_ja_{j+1}\cdots a_k \\ 
&= g_{j+1}a_p \cdots a_k\\
&= g_{j+1}\Delta
\end{align*}

If $j = k$, then
\begin{align*}
\Delta g_k &= a_p\cdots a_{k-1}a_k g_k \\
&= a_p\cdots a_{k-1}g_k^{-1}a_k \\
&=a_p\cdots a_{k-2}g_k^{-1}g_{k-1}^{-1}a_{k-1}a_k\\
&= g_k^{-1}a_p\cdots a_{k-2}g_{k-1}^{-1}a_{k-1}a_k\\
&= g_k^{-1}a_p\cdots g_{k-1}^{-1}g_{k-2}^{-1}a_{k-2}a_{k-1}a_k\\
&\vdots \\ 
&= g_k^{-1}g_{k-1}^{-1}\cdots g_{p+2}^{-1}a_pg_{p+1}^{-1}a_{p+2}\cdots a_k\\
&= g_p^{-1}\cdots g_k^{-1} \Delta
\end{align*}

If $p < j \leq k$ we have
\begin{align*}
g_j\Delta &= g_ja_p \cdot a_{j-1}a_j\cdots a_k \\
&= a_p\cdots a_{j-2}g_ja_{j-1}a_j\cdots a_k \\
&= a_p\cdots a_{j-1}g_jg_{j-1}a_j\cdots a_k \\ 
&= a_p\cdots a_{j-1}g_ja_jg_jg_{j-1}\cdots a_k \\
&= a_p\cdots a_{j-1}a_jg_j^{-1}g_jg_{j-1}\cdots a_k\\
&= a_p\cdots a_{j-1}a_jg_{j-1}\cdots a_k \\ 
&= a_p \cdots a_kg_{j-1}\\
&= \Delta g_{j-1}
\end{align*}

If $j=p$, then
\begin{align*}
g_p\Delta &= g_pa_pa_{p+1}\cdots a_k \\ 
&= a_pg_p^{-1}a_{p+1}\cdots a_k \\ 
&=a_pa_{p+1}g_p^{-1}g_{p+1}^{-1}\cdots a_k \\ 
&= a_pa_{p+1}g_{p+1}^{-1}\cdots a_k g_p^{-1} \\ 
&\vdots \\ 
&= a_p\cdots g_{k-1}^{-1}a_k g_p^{-1}\cdots g_{k-2}^{-1} \\
&= a_p\cdots a_k g_p^{-1}\cdots g_{k}^{-1}\\
&= \Delta g_p^{-1}\cdots g_k^{-1}
\end{align*}

\end{proof}

 \begin{theorem}\label{normal}
 $N$ is a normal abelian subgroup of $\tilde{A}_n$ and 
 \begin{equation}\label {quotion1}
 \faktor{\tilde{A}_n}{N} \cong A_n
 \end{equation}
 \end{theorem}
 
 \begin{proof}
 First we establish that the group $N$ is abelian.  Let $(j_1,i_1)$ and $(j_2,i_2)$ be any pair of indices with $j_k,i_k \in \{1,2,\ldots,n+1\}$. we will show that $p_{j_1}^{-1}p_{i_1}$ and $p_{j_2}^{-1}p_{i_2}$ commute.     If $j_1 = i_1$ or $j_2=i_2$, then there is nothing to show, since the corresponding product is 1. Otherwise, we have by Corollary \ref{inverse pi relations}
\begin{align*}
p_{j_1}^{-1}p_{i_1}\cdot p_{j_2}^{-1}p_{i_2} & = p_{j_1}^{-1} p_{j_2+1}^{-1}p_{i_1-1}p_{i_2}\\
 & = p_{j_1}^{-1} p_{j_2+1}^{-1}p_{i_2-1}p_{i_1}\\
 & = p_{j_2}^{-1} p_{j_1+1}^{-1}p_{i_1-1}p_{i_2}\\
 & = p_{j_2}^{-1} p_{i_2}\cdot p_{j_1}^{-1}p_{i_1}
 \end{align*}
 as claimed.
 
 Let 
 $$P: \tilde{A}_n \to \faktor{\tilde{A}_n}{N}, \ Pw=wN $$ 
 be the quotion homomorphism. To prove \eqref{quotion1} we show that the restriction of $P$ to $\langle a_1,...,a_n\rangle$,  is an isomorphism.
 
 Consider the collection of cosets
 $$\{a_1N,a_2N,\ldots, a_nN\}$$

 We want to show that this collection is enough to generate the quotient.  Recall that $p_{n+1} = a_1\cdots a_n$, so by construction it does not contain the generating element $a_{n+1}$ in it.  In fact, 
 \begin{align*}
 p_{n+1}^{-1}p_1 & = p_{n+1}^{-1}a_2a_3\cdots a_{n_1}a_na_{n+1}
 \end{align*}
 
 This implies 
 $$a_{n+1} = a_na_{n-1}\cdots a_3a_2p_{n+1}(p_{n+1}^{-1}p_1)$$
 Notice that everything to the left of $p_{n+1}^{-1}p_1$ in the last expression is a product of elements $a_1, \ldots, a_n$ only.  This shows that $a_{n+1}$ can be represented as a product of the first $n$ generators and a group element of $N$.  It follows that every element of the quotient can be represented as a product of the cosets described above, and, therefore, the restriction of $P$ to $\langle a_1,...,a_n\rangle$ is onto the quotient. To show that it is one-to-one we must prove that
 $$\langle a_1,...,a_n\rangle \cap H=\{1\}. $$

Since $H$ is a normal subgroup,  all conjugacy classes of elements of $H$ are in $H$. Hence, if $w\in \langle a_1,...,a_n\rangle$ and $w\in H$, the $A_n$-block echelon form of $w$, \ $w^\prime$, is in $H$. Write
$$w^\prime=\Delta_1\cdots \Delta_k.$$
If the block $\Delta_1$ is a single letter word, $\Delta_1=a_p$, by Proposition \ref{ai gj relations} $g_p$ commutes with all $a_j, \ j\geq p+2$, and, hence, $g_p^{-1}$ commutes with $\Delta_2,...,\Delta_k$, and we obtain
$$g_pw^\prime=g_pa_p\Delta_2\cdots\Delta_k=g_p^{-1}a_p\Delta_2\cdots\Delta_k=a_2\Delta_2\cdots \Delta_kg_p^{-1}=w^\prime g_p^{-1}\neq w^\prime g_p, $$
so that $w^\prime$ does not commute with $g_p$, and, thus, can not belong to $H$.

If $|\Delta_1|\geq 2$ write $\Delta_1=a_p\cdots a_s$. We have
\begin{eqnarray*}
&w^\prime \in H \implies a_{p+1}w^\prime a_{p+1} \in H,\\
&a_{p+1}w^\prime a_{p+1}=a_{p+1}a_pa_{p+1}\cdots a_s \Delta_2 \cdots \Delta_k a_{p+1}=\Delta_1 \cdots \Delta_k a_pa_{p+1}=w^\prime a_pa_{p+1}, \\
&\implies a_pa_{p+1}\in H.
\end{eqnarray*}
Now, we have
$$g_p\big(a_pa_{p+1}\big)=a_pg_p^{-1}a_{p+1}=a_pa_{p+1}g_p^{-1}g_{p+1}^{-1}\neq \big(a_pa_{p+1}\big)g_p, $$
a contradiction. We are done.

\end{proof}

 %%%%%%%%%%%%%%%%% Main Result%%%%%%%%%%%%%%%%%%%%%%

 \section{Main result for $\tilde{A}_n$} \label{main}
 
 First we establish a special way to represent each element of $\tilde{A}_n$.
 % in what we call an echelon form using admissible transformations.  
 
 \begin{lemma} \label{echelon form tildean}
 Given an element, $w\in \tilde{A}_n$, using admissible transformations we can transform it into the following form:
 \begin{equation}\label{tilde echelon}
 \Delta_1\cdots \Delta_k g_1^{\ell_1}\cdots g_{i_1}^{\ell_{i_1}}g_{i_2-1}^{\ell_{i_2-1}}g_{i_2}^{\ell_{i_2}}\cdots g_{i_k}^{\ell_{i_k}}g_r^{\ell_r}\cdots g_{n}^{\ell_{n}}
 \end{equation}
 where $i_j$ enumerates the index of the first letter in the $i$-th  block and $r$ is 1 more than the index of the last letter in $\Delta_k$. We call \eqref{tilde echelon}  \textbf{$\tilde{A}$-echelon form} (or simply ``echelon form" when there is no ambiguity) of $w$. 
 \end{lemma}
 
 \begin{proof}
 Let $w\in \tilde{A}_n$. Since $\tilde{A}_n = A_n \rtimes N$, $w$ can be written as 
 $$w= xm \quad x\in A_n, m\in N.$$

Using admissible transformations we can bring $x$ to $A_n$-block echelon form.  This process might change the element $m$, to a different  element  $m^\prime \in N$, so we obtain
$$w = xm \sim \Delta_1\cdots\Delta_k \prod_{j=1}^{n} g_j^{\ell_j}$$
(here
$m' = \prod_{j=1}^{n} g_j^{\ell_j}$).

We perform further admissible transformations to construct our $\tilde{A}_n$-echelon form.  We start with the first block $\Delta_1$.  

(Case 1) $|\Delta_1| = 1$, then $\Delta_1 = a_p$ for some index $1\leq p \leq n$.  We consider $g_p^{\ell_p}$ and if $|\ell_p|<2$ then we do nothing, otherwise we perform the following transformations:
\begin{align*}
 &a_p\Delta_2\cdots\Delta_k g_1^{\ell_1}\cdots g_p^{\ell_p}\cdots g_{n}^{\ell_{n}} \\ 
=& a_p\Delta_2\cdots\Delta_k g_1^{\ell_1}\cdots g_p^{\ell_p-1}\cdots g_{n}^{\ell_{n}}g_p \\ 
\sim & g_pa_p\Delta_2\cdots\Delta_k g_1^{\ell_1}\cdots g_p^{\ell_p-1}\cdots g_{n}^{\ell_{n}} \\
= &a_pg_p^{-1}\Delta_2\cdots\Delta_k g_1^{\ell_1}\cdots g_p^{\ell_p-1}\cdots g_{n}^{\ell_{n}}\\
= &a_p\Delta_2\cdots\Delta_k g_1^{\ell_1}\cdots g_p^{\ell_p-2}\cdots g_{n}^{\ell_{n}}.
\end{align*}

Here we use circular transformation to move one of the $g_p$ elements to the front then commuting relations from Proposition \ref{ai gj relations} and Corollary \ref{commutations delta and g} in the last two lines.  This argument works the same if $\ell_p <0$ as well, so in both cases we have reduced the number of $g_p$ (or $g_p^{-1}$) in this form.  We continue this process until either $\ell_p = 0$ or $\ell_p = \pm 1$.  
 
A similar argument will work for every block of length 1.

(Case 2) $|\Delta_1| > 1$. Write $\Delta_1 = a_p\cdots a_h$.  If $\sum_{i=p+1}^h |\ell_i| =0$ then there is nothing to do, otherwise assume there is at least one non-zero ${\ell_j}$. Let $j=max\Big\{ p\leq i\leq h: \ |\ell_i|\neq 0 \Big \}$.  We perform the following transformations:
\begin{align*}
&\Delta_1 \Delta_2\cdots \Delta_k g_1^{\ell_1}\cdots g_j^{\ell_j}\cdots g_{n}^{\ell_{n}}\\ 
= &\Delta_1 \Delta_2\cdots \Delta_k g_1^{\ell_1}\cdots g_{n}^{\ell_{n}}g_j^{\ell_j}\\
\sim &g_j^{\ell_j}\Delta_1 \Delta_2\cdots \Delta_k g_1^{\ell_1}\cdots g_{n}^{\ell_{n}}\\ 
= &\Delta_1g_{j-1}^{\ell_j} \Delta_2\cdots \Delta_k g_1^{\ell_1}\cdots g_{n}^{\ell_{n}}\\
 = & \Delta_1\Delta_2\cdots \Delta_kg_{j-1}^{\ell_j} g_1^{\ell_1}\cdots g_{n}^{\ell_{n}}\\
 = &\Delta_1 \Delta_2\cdots \Delta_k g_1^{\ell_1}\cdots g_{j-1}^{\ell_{j-1}+\ell_j} \cdots g_{n}^{\ell_{n}}.
\end{align*}
 Here we moved $g_j^{\ell_j}$ to the right, then circularly to the front of the word.  Then we moved $g_j^{\ell_j}$ to the right using the commuting relations in Corollary \ref{commutations delta and g}. We call such transformation \textit{clockwise transformation}.
 
  This process results in decreasing the maximal index with non-trivial power of $g$ elements by at least 1. 
  We  continue to perform clockwise transformations until either there are no $g_i$-s left, or $g_p$ is the only one with a non-trivial power and we reach the word 
  
 $$a_p\cdots a_h \Delta_2\cdots \Delta_k g_1^{\ell_1}\cdots g_p^{\ell_p'}g_{h+1}^{\ell_{h+1}} \cdots g_{n}^{\ell_{n}}$$
 where $\ell_p' = \sum_{i=p}^{h} \ell_i$.  We do this for each block, and leave the remaining $g_j$ elements untouched. The proof is comlete.
 \end{proof}

Let $K$ be a subset of $\tilde{A}_n$ that consists of all $A_n$-block echelon forms   with blocks of at most size $2$ (we consider $A_n=\langle a_1,...,a_n\rangle$), $g_i$ elements, and their inverses.  Note that $K$ is a finite set in $\tilde{A}_n$.
 
  \begin{theorem}\label{main new}
 Suppose that $\rho_1$ and $\rho_2$ are complex finite-dimensional representations of $\tilde{A}_n$.  If we have 
 $$\sigma_p^d(K,\rho_1) = \sigma_p^d(K, \rho_2)$$
 Then $\chi_{\rho_1} = \chi_{\rho_2}.$
 \end{theorem}
 
 \begin{proof}
 Given any element $w\in \tilde{A}_n$ we use admissible transformations to transform it into its $\tilde{A}$-echelon form \eqref{tilde echelon}.  If for every block in \eqref{tilde echelon} $|\Delta_j| \leq 2$, then $\Delta_1\cdots\Delta_k \in K$ and we apply Lemma \ref{Powers lemma} to the elements 
 $$\tau=\Delta_1\cdots\Delta_k \ \mbox{and} \  g_i^{sign(\ell_i)} \ \mbox{for those} \ i\mbox{ where } \ \ell_i \neq 0$$
  and consider the signatures where the multiplicities of $g_i^{sign(\ell_i)}$ equal to $|\ell_i|$ and of $\tau$ - to 1.  It is easy to show by using the commutation relations of Corollary \ref{commutations delta and g} and circular transformations  that every permutation of these elements is in the same conjugacy class as $w$, 
and, therefore, Lemma \ref{Powers lemma} implies that $\chi_{\rho_1}(w)=\chi_{\rho_2}(w)$. 
 
 \vspace{.2cm}
 
Now, let us assume that in the $\tilde{A}$-echelon form of $w$ there are $\Delta_j$ such that $|\Delta_j|\geq 3$. Let $q$ be the number of blocks with $|\Delta_j|\geq 3$ and $i_1,...i_q$ be the indexes of these blocks. We further write $\Delta_{i_j} = a_{p_{i_j}}\cdots a_{k_{i_j}}, \ 1 \leq j \leq q$. Since $\sigma_p^d(K,\rho_1) = \sigma_p^d(K, \rho_2)$, for every $K_1\subset K$ we have 
$$\sigma_p^d(K_1,\rho_1) = \sigma_p^d(K_1, \rho_2)$$ 
as well. We will define a subset $K_1\subset K$ in such way that an application of Lemma \ref{Powers lemma} would imply the quality $\chi_{\rho_1}(w)=\chi_{\rho_2}(w)$.

\vspace{.3cm}

\noindent Step 1. \textbf{Choosing  $K_1$.}

\noindent a). If $i_1>1$ then the word $\Delta_1...\Delta_{i_1-1}$ is in $K$, so we choose $s_1=\Delta_1...\Delta_{r-1}a_{p_{i_1}}a_{p_{i_1+1}}$ as our first element of $K_1$.  If $i_1=1$, then $s_1=a_{p_{i_1}}a_{p_{i_1+1}}$. 

\vspace{.2cm}

\noindent b). We distinguish between the following 2 cases: \\
\noindent $b_1)$.\  $ |\Delta_{i_1}|=3$.  

\noindent If $i_2>i_1+1$, then $s_2=a_{p_{i_1+2}}\Delta_{i_1+1}...\Delta_{i_2-1}a_{p_{i_2}}a_{p_{i_2+1}}$. \\
If the block $\Delta_{i_1+1}$ has length greater than or equal to 3, that is $i_2=i_1+1$, then $s_2=a_{p_{i_1+2}}a_{p_{i_2}}a_{p_{i_2+1}}$. \\
\noindent $b_2$). \ $|\Delta_{i_1}|>3$. In this case we add to $K_1$ single letter words $a_{p_{i_1+2}}, a_{p_{i_1+3}},...,a_{p_{k_{i_1}-1}}$  and the word $s_2$ which is formed by the same rule as in $b_1$). 

\vspace{.2cm}

\noindent c). We proceed adding elements to $K_1$ applying the same rules as in $b_1)$ and $b_2)$ to the subsequent blocks.

\vspace{.2cm}

\noindent d). We finish  the formation of the set $K_1$ by adding to $K_1$ elements $\hat{g}_j=g_j^{sign(\ell_j)}$ for those $j$ where  $\ell_j\neq 0$ in $\tilde{A}$-echelon form  \eqref{tilde echelon} of $w$.
\vspace{.2cm}

\noindent Write  $m_j=|\ell_j|$. The $\tilde{A}$-echelon form \eqref{tilde echelon} of $w$ is written in the alphabet $K_1$ as
\begin{equation}\label{echelon in K1}
s_1\tau_1 s_2\tau_2\cdots 
	s_q \tau_qs_{q+1}\hat{g}_1^{m_1}\cdots \hat{g}_{i_1}^{m_{i_1}}\hat{g}_{i_2-1}^{m_{i_2-1}}\hat{g}_{i_2}^{m_{i_2}}\cdots \hat{g}_{i_k}^{m_{i_k}}g_r^{\ell_r}\cdots \hat{g}_{n}^{m_{n}},
\end{equation}
  
where $\tau_j=a_{p_{i_j+2}}\cdots a_{p_{k_j-1}}$ are words comprised of the letters added to the alphabet $K_1$ following  $b_2)$ rule above.
\vspace{.2cm}

\noindent Step 2. \textbf{Application of Lemma \ref{Powers lemma}}

\vspace{.3cm}

\noindent We will be applying Lemma \ref{Powers lemma} to the set of matrices $\rho_j(K_1), \ j=1,2$ and  the words with signature where the multiplicity  of every element which is not $\hat{g}_j$ is 1 and the multiplicity of $\hat{g}_j$ is $m_j$.  To complete the proof it suffices to show that all words obtained as permutations of letters in \eqref{echelon in K1} with the the same signature lie in the same conjugacy class (note that we admit permutations of the letters the words, $\tau_j$ are comprised of). 

First, we remark that every $\hat{g}_i$ commutes with all single letter elements of $K_1$ and with all but 1 elements of $s$ type elements. This implies that given any permutation we can use admissible transformations bringing it to the form where all $\hat{g}_i$ elements are at the end of the word following single letter and $s$ type elements. If the element which does not commute with $\hat{g}_i$ precedes it, then we move $\hat{g}_i$ to the right; if it follows  $\hat{g}_i$, we move $\hat{g}_i$ to the left to the beginning of the word, and then use a circular transformation to move it to the back of the word. Now, our word looks like 
\begin{equation}\label{1} 
\delta_1s_1\delta_2\hat{g}_1^{m_1}\cdots \hat{g}_{i_1}^{m_{i_1}}\hat{g}_{i_2-1}^{m_{i_2-1}}\hat{g}_{i_2}^{m_{i_2}}\cdots \hat{g}_{i_k}^{m_{i_k}}\hat{g}_r^{m_r}\cdots \hat{g}_{n}^{m_{n}},
\end{equation}
where $\delta_1$ and $\delta_2$ do not contain $g$ letters.
We use circular transformation to get
$$
s_1\delta_2\gamma_1 \gamma_2\delta_1,
$$
where $\gamma_1$ is the product of all $\hat{g}$ elements in \eqref{1} that commute with $\delta_1$, and $\gamma_2$ - of those which do not (of course $\gamma_1$ and $\gamma_2$ commute, since both are in $N$). Since every $g$ element in $\gamma_2$ does not commute with $\delta_1$ and for every $g$ element there is only 1 element in $K_1$ which it does not commute with, $\gamma_2$ commutes with $\delta_2$ and $s_1$, so that
\begin{equation}\label{transformations1}
s_1\delta_2\gamma_1 \gamma_2\delta_1 =  \gamma_2 s_1\delta_2 \delta_1 \gamma_1\sim s_1\delta_2 \delta_1\gamma_1\gamma_2. 
\end{equation}
If $\tau_1=1$ (which happens in the case described in $b_1)$), we have transformed our permuted word into
\begin{equation}\label{transformations2}
s_1\tau_1\cdots.	
\end{equation}

Now, suppose that $\tau_1=a_{p_{i_1+2}}\cdots a_{p_{k_1-1}}\neq 1$. All the letters of which $\tau_1$ is comprised of, excepts for $a_{p_{i_1+2}}$, commute with $s_1$. Let us call $a_{p_{i_1+s_1}},\cdots ,a_{p_{i_1+s_t}}$ those of them which lie between $s_1$ and $a_{p_{i_1+2}}$ in \eqref{transformations1}. We also observe that $a_{p_{i_1+2}},\cdots,,a_{p_{k_1-2}} $ commute with $s_j$ and all letters in $\tau_j$ for $j\geq 2$ and that $a_{p_{k_1-1}}$ commutes with all $s_j, \ j\geq 3$ and all letters in $\tau_j, \ j\geq 2$. 

\vspace{.2cm}

\noindent I). If $a_{p_{k_1-1}}$ is not among $a_{p_{i_1+s_1}},\cdots ,a_{p_{i_1+s_t}}$, or, if $a_{p_{k_1-1}}$ is among them, but $s_2$ is not between $s_1$ and $a_{p_{k_1-1}}$, we can move these letters one after another to the very beginning of the word in front of $s_1$, and then using a circular transformation bring them to the back, thus getting a word
$$s_1\cdots a_{p_{i_1+2}}\cdots, $$
and here all letters of the alphabet $K_1$ that are between $s_1$ and $a_{p_{i_1+2}}$ commute with $a_{p_{i_1+2}}$, so that we can move $a_{p_{i_1+2}}$ to the left resulting in
\begin{equation}\label{transformations3} 
s_1a_{p_{i_1+2}}\cdots . 
\end{equation}

\vspace{.2cm}

\noindent II). If $a_{p_{k_1-1}}$ is among $a_{p_{i_1+s_1}},\cdots ,a_{p_{i_1+s_t}}$ and $s_2$ is between $s_1$ and $a_{p_{k_1-1}}$, we first move $s_2$ to the left, in the very front of the word (as it commutes with every letter in front of it), then circularly move it to the back of the word, and, finally, proceed as in I) again resulting in \eqref{transformations3}.

\vspace{.2cm}

We proceed this way with letters $a_{p_{i_1+3}},...a_{p_{k_1-1}}$ until we obtain a word

$$
s_1a_{p_{i_1+2}}\cdots a_{p_{k_1-1}}\cdots=s_1\tau_1\cdots . 
$$

We continue inductively. Suppose that we have already transformed our word to the form
\begin{equation}\label{transformations4} 
s_1\tau_1\cdots s_t\tau_t \cdots 
\end{equation} 

Consider the word comprised of all the elements of the alphabet $K_1$ that are between $\tau_t$ and $s_{t+1}$ in \eqref{transformations4}. It might be written as
$$\delta_1 s_{u_1} \delta_2 s_{u_2}\cdots s_{u_l}\delta_{l+1}, $$
where $\delta_1,\cdots,\delta_{l+1}$ do not contain $s$-letters. Since all $\hat{g}$ elements commute with all $a$ letters in $K_1$, using only commutation transformations the word $\delta_1$ can be transformed into
$$\gamma_1 \gamma_2 \alpha, $$ 
where $\gamma_1$ and $\gamma_2$ are comprised only of $\hat{g}$ letters, $\alpha$ contains no $\hat{g}$ letters ( and no $s$ letters), $\gamma_1$ commutes with $s_1\tau_1\cdots s_t\tau_t$, and $\gamma_2$ does not. 
We move $\gamma_1$ to the left, and then circularly to the end of the word to obtain
\begin{equation}\label{transformations5}
s_1\tau_1\cdots s_t\tau_t \gamma_2 \alpha s_{u_1}\cdots s_{u_l}\delta_{l+1} s_{t+1}\cdots \gamma_1.
\end{equation}
Since $\gamma_2$ does not commute with $s_1\tau_1\cdots s_t\tau_t$ and for each $\hat{g}$ there is only one $s$ element it does not commute with, $\gamma_2$ commutes with all elements in \eqref{transformations5} which follow it, and moving it to the right we obtain
$$s_1\tau_1\cdots s_t\tau_t \alpha s_{u_1}\cdots s_{u_l}\delta_{l+1} s_{t+1}\cdots \gamma_2 \gamma_1. $$
Further, all letters in $\alpha$ commute with $s_1\tau_1\cdots s_t\tau_t$, so moving it to the left, and circularly we obtain
$$s_1\tau_1\cdots s_t\tau_t  s_{u_1}\cdots s_{u_l}\delta_{l+1} s_{t+1}\cdots \alpha \gamma_2 \gamma_1. $$
Similarly, if $u_1\neq t+1$, \ $s_{u_1}$ commutes with $s_1\tau_1 \cdots s_t\tau_t$ and in a similar way we can transform our word to
$$s_1\tau_1\cdots s_t\tau_t \delta_2s_{u_2}  \delta_{l+1} s_{t+1}\cdots \alpha \gamma_2 \gamma_1 s_{u_1}. $$
We apply the similar procedure to move $\delta_2 s_{u_2},\cdots , \delta_ls_{u_l}$ and $\delta_{l+1}$ to the end of the word and yield
$$s_1\tau_1 \cdots s_t\tau_t s_{t+1} \cdots . $$
The procedure of moving letters of in $\tau_{t+1} $ to directly follow $s_{t+1}$ is similar to the one above described in I) and II). As a result we have transformed our word into
$$s_1\tau_1\cdots s_t\tau_t s_{t+1}\tau_{t+1}\cdots , $$
finishing the induction step.

After $q$ steps we arrive to the word \eqref{echelon in K1} as required. We are done.

\end{proof}

\end{document}